\documentclass[12pt]{article}
\usepackage{amssymb,amsmath,color}
\usepackage{hyperref}
\usepackage{graphicx}
\usepackage{epsfig}

\newtheorem{theorem}{Theorem}[section]

\newtheorem{corollary}[theorem]{Corollary}

\numberwithin{equation}{section}

\newcommand{\proof}{\noindent {\bf Proof:~}}
\newcommand{\qed}{\hfill {\bf QED}}

\newcommand{\R}{\mathbb{R}}
\newcommand{\Z}{\mathbb{Z}}
\newcommand{\T}{\mathbb{T}}
\newcommand{\C}{\mathbb{C}}

\newcommand{\eps}{\varepsilon}

\renewcommand{\phi}{\varphi}

\newcommand{\FD}[2]{\frac{d {#1}}{d {#2}}}

\begin{document}

\title{Hopf normal form with $S_N$ symmetry and reduction to systems of nonlinearly coupled phase oscillators}

\author{Peter Ashwin\thanks{Corresponding author: {\tt p.ashwin@exeter.ac.uk}, +441392725225} and Ana Rodrigues\\Department of Mathematics\\University of Exeter\\Exeter EX4 4QF, UK}

\maketitle

\begin{abstract}
Coupled oscillator models where $N$ oscillators are identical and symmetrically coupled to all others with full permutation symmetry $S_N$ are found in a  variety of applications. Much, but not all, work on phase descriptions of such systems consider the special case of pairwise coupling between oscillators. In this paper, we show this is restrictive - and we characterise generic multi-way interactions between oscillators that are typically present, except at the very lowest order near a Hopf bifurcation where the oscillations emerge. We examine a network of identical weakly coupled dynamical systems that are close to a supercritical Hopf bifurcation by considering two parameters, $\epsilon$ (the strength of coupling) and $\lambda$ (an unfolding parameter for the Hopf bifurcation). For small enough $\lambda>0$ there is an attractor that is the product of $N$ stable limit cycles; this persists as a normally hyperbolic invariant torus for sufficiently small $\epsilon>0$. Using equivariant normal form theory, we derive a generic normal form for a system of coupled phase oscillators with $S_N$ symmetry. For fixed $N$ and taking the limit  $0<\epsilon\ll\lambda\ll 1$, we show that the attracting dynamics of the system on the torus can be well approximated by a coupled phase oscillator system that, to lowest order, is the well-known Kuramoto-Sakaguchi system of coupled oscillators. The next order of approximation genericlly includes terms with up to four interacting phases, regardless of $N$. Using a normalization that maintains nontrivial interactions in the limit $N\rightarrow \infty$, we show that the additional terms can lead to new phenomena in terms of coexistence of two-cluster states with the same phase difference but different cluster size.
\end{abstract}

\section{Introduction}

Coupled oscillator models are used in a wide variety of applications. They appear in neuroscience for studying neuronal oscillation patterns in the brain (see for example \cite{AshCooNic16,Farmer02,Gray94,Sin93}); in chemistry (see for example \cite{TTZS09,TNS12}) and in physics (see for example \cite{Acebron05,WS95}). A powerful method for understanding the dynamics of coupled oscillators comes from noting that $N$ limit cycle oscillators give rise to a normally hyperbolic invariant torus that persists for weak enough coupling \cite{AshSwi91}. In such cases we can describe the asymptotic dynamics in terms of just phases. A specific example of a coupled identical phase oscillator system with global (all-to-all) coupling is that of Kuramoto \cite{Kuramoto1975}
\begin{equation}
\FD{}{t} {\phi}_j = \omega + \frac{K}{N} \sum_{k=1}^{N} g(\phi_k-\phi_j),
\label{eq:TNkuramoto}
\end{equation}
with fixed natural frequency $\omega$, coupling strength $K>0$. Although the original work of Kuramoto considered $g(\phi)=\sin(\phi)$, a more general ``Kuramoto-Sakaguchi'' coupling (phase interaction) function \cite{sakaguchi_kuramoto_86} is
\begin{equation}
g(\phi) = \sin(\phi+\alpha).
\label{eq:kuramotocoupling}
\end{equation}
For the system (\ref{eq:TNkuramoto},\ref{eq:kuramotocoupling}), the only attractors are full synchrony or full asynchrony, depending on the value of the parameter $\alpha$, while in the special case $\cos(\alpha)=0$ the system is integrable. Many papers have studied the dynamics of this and related systems; see for example \cite{Acebron05,Strogatz2000}. For this permutation symmetric case of identical oscillators, the system above is known to behave in ways that are not generic, even accounting for symmetries. There can be a large number of integrals of the motion \cite{Watanabe1994} and degenerate bifurcation behaviour on varying $\alpha$ \cite{AshBur2015}. As pointed out in \cite{KorKruJaiKisHud14}, for weak linear coupling of nonlinear systems near Hopf bifurcation, one expects to have a coupling function $g$ that is smooth and $2\pi$-periodic \cite{AshSwi91,Dai96}. That is, the generic situation is that all $A_k$ will be non-zero in the Fourier expansion
\begin{equation}
g(\phi)= \sum_{k=0}^{\infty} A_k \sin(k\phi+\chi_k)
\end{equation}
where the $A_k$ will decay with $k$ at a rate that will depend on the smoothness of $g$.

However, recent work by Rosenblum, Pikovsky and co-workers has highlighted that more complex interactions may be present in coupled oscillator systems, and that this may lead to new emergent phenomena such as self-organized quasiperiodicity \cite{RosPik2007} on including an additional damped equation, or on including direct dependence of a phase shift $\chi_k$ on an order parameter \cite{BurPik11,PikRos2009}.

The current paper considers generic nonlinear, but fully permutation symmetric, weak coupling of $N$ identical Hopf bifurcations. We show, by examining a generic normal form for equivariant Hopf bifurcation and unfolding parameter $\lambda$, that the system has an attracting invariant torus for  $0<\epsilon\ll \lambda\ll 1$. On this torus, the flow can be approximated by (\ref{eq:TNkuramoto},\ref{eq:kuramotocoupling}) at lowest order, but over a longer timescale it can be better approximated by a system of the form
\begin{equation}\label{principalEQ}
\begin{split}
\dfrac{d}{dt} {\phi}_j = & \tilde{\Omega}(\phi,\epsilon)  + \frac{\epsilon}{N}\sum_{k=1}^N g_2(\phi_k-\phi_j)+\frac{\epsilon}{N^2}\sum_{k,\ell=1}^{N}g_3(\phi_k+\phi_{\ell}-2\phi_j) \\ 
& + \frac{\epsilon}{N^2}\sum_{k,\ell=1}^{N} g_4(2\phi_k-\phi_{\ell}-\phi_j) + \frac{\epsilon}{N^3}\sum_{k,\ell,m=1}^{N} g_5(\phi_k+\phi_{\ell}-\phi_{m}-\phi_j).
\end{split}
\end{equation}
The frequency $\tilde{\Omega}(\phi,\epsilon)$ is a symmetric function of the phases that is close to the frequency at Hopf bifurcation of the uncoupled system, and we have coupling via
\begin{equation}
\begin{split}
g_2(\phi)= & \xi_1 \cos (\phi+\chi_1)+ \xi_2\cos (2\phi+\chi_2)\\
g_3(\phi)= & \xi_3 \cos (\phi+\chi_3)\\
g_4(\phi)= & \xi_4 \cos (\phi+\chi_4)\\
g_5(\phi)= & \xi_5 \cos (\phi+\chi_5)
\end{split}
\label{eq:coupform}
\end{equation}
where $\xi_j$ and $\chi_j$ depend on $\lambda$. More precisely, they are determined by 
\begin{equation}
\begin{split}
g_2(\phi)= & \xi_1^0 \cos (\phi+\chi_1^0)+\lambda \xi_1^1 \cos (\phi+\chi_1^1)+ \lambda\xi_2^1\cos (2\phi+\chi_2^1)\\
g_3(\phi)= & \lambda\xi_3^1 \cos (\phi+\chi_3^1)\\
g_4(\phi)= & \lambda\xi_4^1 \cos (\phi+\chi_4^1)\\
g_5(\phi)= & \lambda\xi_5^1 \cos (\phi+\chi_5^1)
\end{split}
\label{eq:coupform2}
\end{equation}
for some constant coefficients $\xi_i^j$ and $\chi_i^j$. A more precise statement that includes the suppressed higher order terms is given in Theorem~\ref{thm:main} and Corollary~\ref{cor:main}. Most of the discussion, apart from Section~\ref{sec:twocluster}, will assume $N$ is fixed, but we assume the given normalization in (\ref{principalEQ}) of the sums by $N$, $N^2$ or $N^3$ to ensure non-trivial coupling in the thermodynamic limit $N\rightarrow\infty$. 

Including only the very lowest order terms, we will see that (\ref{principalEQ}) reduces to (\ref{eq:TNkuramoto}) with coupling (\ref{eq:kuramotocoupling}):
\begin{equation}\label{principalEQlambdazero}
\dfrac{d}{dt} {\phi}_j = \Omega  + \frac{\epsilon}{N}\sum_{k=1}^N \xi_1^0 \cos (\phi_k-\phi_j+\chi_1^0). 
\end{equation}
with $\Omega$, $\xi_1^0$ and $\chi_1^0$ constants. As (\ref{principalEQ})  shows, to the next order we only need to consider interaction terms of up to four phases. Each of the smooth periodic functions $g_k(\phi)$ for $k=1,\ldots,5$ involves only one Fourier mode, except for $g_2$ which involve two, and $\tilde{\Omega}(\phi,\epsilon)$ is a symmetric function of the phases. Note that \eqref{principalEQ} has $S^1$ normal form symmetry
$$
(\phi_1,\ldots,\phi_N) \mapsto (\phi_1+\chi,\ldots,\phi_N+\chi) 
$$
for any $\chi\in\T$, in addition to the permutation symmetries $S_N$.

The structure of the paper is as follows. In Section~\ref{sec:equivariant} we give an outline of the normal form theory for $S_N$ equivariant Hopf bifurcation on $\C^N$ where $S_N$ acts naturally by permutation of coordinates. This action decomposes into two irreducible subspaces of complex type, one of dimension one (corresponding to bifurcation to in-phase oscillation) and one of dimension $N-1$ (corresponding to bifurcation to anti-phase oscillation). We include two bifurcation parameters, $\lambda$ determining the Hopf bifurcation $\epsilon$ representing the strength of coupling, in regimes where both are small.

Section~\ref{sec:reduction} considers a set of coupled systems undergoing a generic supercritical Hopf bifurcation in the case of weak coupling $0<\epsilon\ll\lambda\ll 1$. The main result is Theorem~\ref{thm:main} which is proved in Section~\ref{sec:mainproof} by performing an explicit reduction of the normal form to an invariant $N$-torus represented by coupled phase oscillators. Section~\ref{sec:examples} briefly considers a numerical example of the reduction as well as discussions of the consequences of the new interaction terms on fully synchronous and two-cluster states. The new terms introduce a particular (quadratic) nonlinearity to the equations for the phases of two cluster states, and in Theorem~\ref{thm:twoclusters} we detail a particular new phenomenon. Finally, in Section~\ref{sec:discussion} we discuss some implications of this on the dynamics of all-to-all coupled oscillators near Hopf bifurcation, and we relate to other work in the literature that considers more general nonlinear coupling between oscillators.

\section{Equivariant Hopf bifurcation with $S_N$ symmetry}
\label{sec:equivariant}

Suppose we have $N$ identical and identically interacting smooth ($C^{\infty}$) dynamical systems on $x_i\in\R^d$ ($d\geq 2$), generated by the following coupled ordinary differential equations:
\begin{eqnarray}
\FD{}{t} {x}_1 &=& H_{\lambda}(x_1)+\epsilon h_{\lambda,\epsilon}(x_1;x_2,\ldots,x_N)\nonumber\\
\vdots & \vdots & \vdots \label{eq:Rdsystem}\\
\FD{}{t} {x}_N &=& H_{\lambda}(X_N)+\epsilon h_{\lambda,\epsilon}(x_N;x_1,\ldots,x_{N-1}).\nonumber
\end{eqnarray}
The ``coupling parameter'' $\epsilon\in\R$ is such that the system completely decouples for $\epsilon=0$. We also assume that each system undergoes a Hopf bifurcation of $x=0$ when a ``Hopf parameter'' $\lambda\in\R$ passes through zero for $\epsilon=0$. 

Without loss of generality we assume that the uncoupled system for $x\in\R^d$ given by
\begin{equation}
\FD{}{t} {x} =  H_{\lambda}(x),
\end{equation}
has a linearly stable fixed point at $x=0$ for $\lambda<0$ that undergoes Hopf bifurcation at $\lambda=0$. Without loss of generality, we can assume that $DH_{\lambda}(0)$ has a complex pair of eigenvalues
$$
\lambda\pm i\omega
$$
where $\omega\neq 0$ and all other eigenvalues $\mu$ of $DH_{\lambda}(0)$ satisfy $Re(\mu)<-r<0$.

We also assume that, without loss of generality, $(x_1,\ldots,x_N)=0$ is an equilibrium for $(\lambda,\epsilon)$ in some neighbourhood of $(0,0)$. As we will be interested in attracting behaviour near bifurcation we assume that the bifurcation is supercritical, i.e. it gives rise to a small amplitude stable limit cycle for $\lambda>0$.

Note that the Jacobian of \eqref{eq:Rdsystem} at $(x_1,\ldots,x_N)=0$ will have the form
\begin{equation}
\left(\begin{array}{cccc}
DH_{\lambda}(0)+\epsilon d_1h_{\lambda,\epsilon}(0) & \epsilon d_2h_{\lambda,\epsilon}(0) & \cdots & \epsilon d_2h_{\lambda,\epsilon}(0) \\
\\
\epsilon d_2h_{\lambda,\epsilon}(0) & DH_{\lambda}(0)+\epsilon d_1h_{\lambda,\epsilon}(0)  & \cdots & \epsilon d_2h_{\lambda,\epsilon}(0) \\
\vdots & \vdots & \ddots & \vdots \\
\epsilon d_2h_{\lambda,\epsilon}(0) & \epsilon d_2h_{\lambda,\epsilon}(0)  & \cdots & DH_{\lambda}(0)+\epsilon d_1h_{\lambda,\epsilon}(0)  \end{array}\right)
\end{equation}
where $d_k$ represents the Jacobian with respect to the $k$th argument.

We assume that the coupling respects the fact that the uncoupled systems can be permuted arbitrarily, i.e. that the system is equivariant under the action of $S_N$ on $\R^{dN}$ by permutation
\begin{equation}\label{eqxaction1}
\sigma (x_1,\ldots, x_N) = (x_{\sigma^{-1}(1)}, \ldots, x_{\sigma^{-1}(N)}),
\end{equation}
for any $(x_1,\ldots,x_N) \in \R^{dN}$ and $\sigma \in S_N$ .

Although Hopf bifurcation in the absence of symmetry can generically be reduced to a two dimensional centre manifold, this is not the case here - the action of the symmetry group $S_N$ means that for $\epsilon>0$ the centre manifold at generic bifurcation will generically be either $2$ dimensional or $2N-2$ dimensional. In the uncoupled case $\lambda=\epsilon=0$ the extra structure means that the centre manifold will be $2N$ dimensional.

\section{Hopf normal form for a weakly coupled system}
\label{sec:reduction}

Using equivariant bifurcation theory \cite{GSS88} it is possible to write the system of ODEs (\ref{eq:Rdsystem}) on a centre manifold $(z_1,\ldots,z_N)\in\C^N$ where in the case $\lambda=\epsilon=0$ the centre manifold in each coordinate $x_k$ is parametrized by $z_k$. This system on the centre manifold is
\begin{eqnarray}
\FD{}{t} {z}_1&=& f_{\lambda}(z_1)+ \epsilon g_{\lambda}(z_1;z_2,\ldots,z_N)+O(\epsilon^2)\nonumber\\
\vdots & \vdots & \vdots \label{eq:CNsystem}\\
\FD{}{t} {z}_N&=& f_{\lambda}(z_N)+ \epsilon g_{\lambda}(z_N;z_1,\ldots,z_{N-1})+O(\epsilon^2)\nonumber
\end{eqnarray}
where $z \in {\mathbb C}^{N}$ and we note the right hand sides can be chosen to be $C^r$, with $r$ arbitrarily large, in a neighbourhood of the bifurcation. The conditions for Hopf bifurcation mean that for \eqref{eq:CNsystem} we have $f_{0}(0)=0$ and $df_{0}(0)$ has a pair of purely imaginary eigenvalues $\pm i\omega$ that pass transversely through the imaginary axis with non-zero speed on changing $\lambda$. The action of $S_N$ on $\C^N$ where $\sigma \in S_N$ acts by permutation of coordinates
\begin{equation}\label{eqaction1}
\sigma (z_1,\ldots, z_N) = (z_{\sigma^{-1}(1)}, \ldots, z_{\sigma^{-1}(N)}),
\end{equation}
where $(z_1,\ldots,z_N) \in \C^N$ and so $g_{\lambda}(z_1;z_2,\ldots,z_N)$ is symmetric under all permutations of the last $N-1$ arguments  that fix the first.

Poincar\'{e}-Birkhoff normal form theory means that to all polynomial orders we can assume there is a normal form symmetry given by the action of $S^{1}$ on $ {\mathbb C}^{N}$
\begin{equation}\label{eqaction2}
\theta (z_1,\ldots, z_N) = e^{i \theta}  (z_1,\ldots, z_N).
\end{equation}

The symmetries \eqref{eqaction1}, \eqref{eqaction2} restrict the possible terms that can appear in the normal form; we can characterise these by finding the possible equivariants, one order at a time. This can be expressed in the following form which can be recovered from \cite{DR05}, where $\sum_{i}$ denotes $\sum_{i=1}^N$, $\sum_{i,j}$ denotes $\sum_{i=1}^N\sum_{j=1}^N$ and  $\sum_{i,j,k}$ denotes $\sum_{i=1}^N\sum_{j=1}^N\sum_{k=1}^N$. 

\begin{theorem}\label{thm:expand}
Suppose $N \geq 4$. Let $f:{\mathbb C}^{N} \rightarrow {\mathbb C}^{N}$
be $S_{N} \times S^{1}$-equivariant with respect
to the action  \eqref{eqaction1}, \eqref{eqaction2} with polynomial components of 
degree lower or equal than $3$. Then $f= (f_1, f_2, \ldots, f_N)$ where
\begin{equation}
\begin{split}
 f_{1}(z_{1},z_{2},\ldots,z_{N})=&
 \sum_{i=-1}^{11}a_{i}h_{i}(z_{1},z_{2},\ldots,z_{N})\\
 f_{2}(z_{1},z_{2},\ldots,z_{N})=&
 f_{1}(z_{2},z_{1},\ldots,z_{N})\\
 & \vdots  \\
 f_{N}(z_{1},z_{2},\ldots,z_{N})=& 
f_{1}(z_{N},z_{2},\ldots,z_{1})
\end{split}
\end{equation}
\noindent and
\begin{equation}\label{eqinv5}
\begin{array}{ll}
h_{-1}(z)~=~\frac{1}{N}\sum_{j}z_{j}, \qquad
& h_{0}(z)~=~z_{1},\\
 h_{1}(z)~=~ |z_{1}|^{2}z_{1} & \\
h_{2}(z)~=~  z_{1}^{2}\frac{1}{N}\sum_{j}{{\overline
z}_{j}},\qquad & 
h_{3}(z)~=~ |z_{1}|^{2}\frac{1}{N}\sum_{i}{z_{i}}, \\
h_{4}(z)~=~  z_{1}\frac{1}{N}\sum_{k}{|z_{k}|^{2}},\quad &
h_{5}(z)~=~
 z_{1}\frac{1}{N^2}\sum_{i,k}{z_{i}}{\overline z}_{k}, \\
h_{6}(z)~=~  {\overline z}_{1}\frac{1}{N}\sum_{j}{
z_{j}^{2}},\qquad & h_{7}(z)~=~  {\overline
z}_{1}\frac{1}{N}\sum_{i,j}{z_{i} z_{j}},\\
h_{8}(z)~=~  \frac{1}{N}\sum_{j}{|z_{j}|^{2}z_{j}}, \qquad &
h_{9}(z)~=~ \frac{1}{N^2} \sum_{i,j}{z_{i}^{2}}{{\overline z}_{k}},\\
h_{10}(z)~=~ \frac{1}{N^2} \sum_{i,k}{z_{i}}{|z_{k}|^{2}}, \qquad &
h_{11}(z)~=~ \frac{1}{N^3} \sum_{i,j,k}{z_{i}}{z_{j}}{{\overline z}_{k}}, \\
\end{array}
\end{equation}
\noindent for constants $a_{j} \in {\mathbb C}$. Also we denote
$|z_j|^2 = z_j \overline{z}_j$ for $j=1, \ldots, N$.
\end{theorem}

\begin{proof} For details, see \cite[Section 2.1.2]{DR05}.
\end{proof}

~

We summarise so far: if system \eqref{eq:Rdsystem} is such that (a) the system decouples for $\epsilon=0$ and (b) for $\epsilon=0$ each system has a generic Hopf bifurcation at $\lambda=0$, $x=0$, then near $\lambda=\epsilon=x=0$ the dynamics can be written on a centre manifold of dimension $2N$ as \eqref{eq:CNsystem}. We now state the main result of our paper:

\newpage

\begin{theorem}\label{thm:main}
Consider system \eqref{eq:CNsystem} with $S_N$-symmetry (for fixed $N$) such that the $N$ uncoupled systems ($\epsilon=0$) undergo a generic supercritical Hopf bifurcation on $\lambda$ passing through $0$. There exists $\lambda_0>0$ and $\epsilon_0=\epsilon_0(\lambda)$ such that for any $\lambda \in (0,\lambda_0)$ and $|\epsilon|<\epsilon_0(\lambda)$ the system \eqref{eq:CNsystem} has an attracting $C^r$-smooth invariant $N$-dimensional torus for arbitrarily large $r$.

Moreover, on this invariant torus, the phases $\phi_j$ of the flow can be expressed as a coupled oscillator system
\begin{equation}\label{principalEQ2}
\begin{split}
\FD{}{t} {\phi}_j = &  \tilde{\Omega}(\phi,\epsilon) + \frac{\epsilon}{N}\sum_{k=1}^N g_2(\phi_k-\phi_j)+\frac{\epsilon}{N^2}\sum_{k,\ell=1}^{N}g_3(\phi_k+\phi_{\ell}-2\phi_j) \\ 
& + \frac{\epsilon}{N^2}\sum_{k,\ell=1}^{N} g_4(2\phi_k-\phi_{\ell}-\phi_j) + \frac{\epsilon}{N^3}\sum_{k,\ell,m=1}^{N} g_5(\phi_k+\phi_{\ell}-\phi_{m}-\phi_j) \\
& +\epsilon \tilde{g}_j(\phi_1,\ldots,\phi_N)+O(\epsilon^2)
\end{split}
\end{equation}
for fixed $0<\lambda<\lambda_0$ in the limit $\epsilon\rightarrow 0$, where $\tilde{\Omega}(\phi,\epsilon)$ is independent of $j$ and 
\begin{equation}
\begin{split}
g_2(\phi)= & \xi_1^0 \cos (\phi+\chi_1^1 ) +\lambda\xi_1^1 \cos (\phi+\chi_1^1)+ \lambda \xi_2^1 \cos (2\phi+\chi_2^1)\\
g_3(\phi)= & \lambda \xi_3^1 \cos (\phi+\chi_3^1)\\
g_4(\phi)= & \lambda \xi_4^1 \cos (\phi+\chi_4^1)\\
g_5(\phi)= & \lambda \xi_5^1 \cos (\phi+\chi_5^1).
\end{split}
\label{eq:coupform2}
\end{equation}
The constants $\xi_i^j$ and $\chi_i^j$ are generically non-zero. 
The error term satisfies
$$
\tilde{g}(\phi_1,\ldots,\phi_N)= O(\lambda^2)
$$
uniformly in the phases $\phi_k$. The truncation of (\ref{principalEQ2}) by removing $\tilde{g}$ and $O(\epsilon^2)$ terms is valid over time intervals $0<t<\tilde{t}$ where $\tilde{t}=O(\epsilon^{-1}\lambda^{-2})$ in the  limit $0< \epsilon\ll\lambda\ll 1$. In particular, for any $N$, this approximation involves up to four interacting phases.
\end{theorem}

The proof of this Theorem is given in the next section. We remark that if we set $\xi_k^1=0$ in the theorem above, this gives the Kuramoto-Sakaguchi system as a truncation but with errors $O(\epsilon\lambda)$, meaning the timescale of validity of the Kuramoto-Sakaguchi system approximation will typically only be $O(\epsilon^{-1}\lambda^{-1})$. We discuss the implications on timescales of validity of the approximation more precisely in the following corollary which is obtained by integrating the $O(\epsilon \lambda^2)$ error term in the truncation.

\begin{corollary}
\label{cor:main}
Consider the system and hypotheses as in Theorem~\ref{thm:main}. Then for any $0<a<1$ and any $0<\epsilon\ll\lambda\ll 1$ such that there is an attracting $N$-torus, there is a timescale $t_{\max}=O(\epsilon^{a-1}\lambda^{-2})$ such that for any solution $\phi(t)$ of (\ref{principalEQ2}), there is a solution $\tilde{\phi}(t)$ of the truncated equation with 
$$
|\phi(t)-\tilde{\phi}(t)|<\epsilon^{a}
$$
over $0<t<t_{\max}$. If we truncate further to only Kuramoto-Sakaguchi terms by setting $\xi^1_k=0$, then this will be possible only over a shorter timescale $t_{\max}=O(\epsilon^{a-1}\lambda^{-1})$.
\end{corollary}

\section{Proof of Theorem~\ref{thm:main}}
\label{sec:mainproof}

We  write the equation for $\FD{}{t} {z}_1$ from \eqref{eq:CNsystem} in Poincar\'{e}-Birkhoff normal form \cite{GSS88} as the $S_N\times S^1$-equivariant system
\begin{equation}
\label{eq:inv5}
\FD{}{t} {z}_1=  U(z_1) + \epsilon F_1(z_1,\ldots,z_N,\epsilon),
\end{equation}
and the equations for the other $\FD{}{t} {z}_j$ are obtained by permutation of the indices; there is an error term but this is beyond all (polynomial) orders.

Since we are assuming there is a Hopf bifurcation of \eqref{eq:CNsystem} for $\epsilon=0$ (the uncoupled system) on varying $\lambda$ through $0$, it follows that 
\begin{equation}
\label{eqhopf}
\FD{}{t} {z_1}= U(z_1):=V(z_1)z_1:=\left[\lambda + i \omega + a_1 |z_1|^{2} + \tau(z_1)\right]z_1,
\end{equation}
and we write $V(z_1)=V_R(z_1)+iV_I(z_1)$. The tail  $\tau(z_1)z_1$ represents the higher order terms in the normal form for the uncoupled system: we can assume $\tau(0)=\tau'(0)=\tau''(0)=0$ and write $\tau(z_1)=\tau_R(z_1)+i\tau_I(z_1)$. The hypothesis that the Hopf bifurcation is generic and supercritical implies
$$
a_{1R}<0.
$$
We seek solutions of (\ref{eqhopf}) of the form
\begin{equation}
z_1(t)=R_1(t)e^{i\phi_1(t)}=R_1(t)e^{i[\Omega t+\psi_1(t)]}
\label{eq:z1po}
\end{equation}
for some $R_1(t)$, $\psi_1(t)$ and constant $\Omega$. Substituting this into (\ref{eqhopf}), we require
$$
\FD{}{t} R_1 + iR_1 \left[\Omega+\FD{}{t} \psi_1\right] = R_1 V_R(R_1)+ i R_1 V_I(R_1).
$$
From (\ref{eqhopf}), note that 
$$
V_R(R_1)=\lambda + a_{1R} R_1^2+ \tau_R(R_1^2),~~V_I(R_1) =\omega+a_{1I} R_1^2+\tau_I(R_1).
$$ 
From this, it is clear that for small enough $\lambda>0$ and $\epsilon=0$ there is a stable periodic orbit at fixed $R_1=R_*>0$ such that $V_R(R_*)=0$, with angular frequency $\Omega=V_I(R_*)$ and arbitrary but fixed phase $\psi_1$.

More precisely, solving $V_R(R_*)=0$, we note (recalling $a_{iR}<0$) that
\begin{equation}
\begin{split}
R_*^2=& \frac{\lambda}{-a_{1R}}+O(\lambda^{2}),\\
\Omega =& V_I(R_*^2)= \omega+ a_{1I} R_*^2+\tau(R_*)= \omega + \frac{a_{1I}}{-a_{1R}} \lambda+O(\lambda^2).
\label{eqRstar}
\end{split}
\end{equation}
In particular there is a $\lambda_0>0$ such that for any $0<\lambda<\lambda_0$ there is a stable periodic orbit (\ref{eq:z1po}) satisfying (\ref{eqRstar}).

Now consider the dynamics of the full (but still uncoupled) system. For $\epsilon=0$ and any $0<\lambda<\lambda_0$  there is a stable invariant torus given by
\begin{equation}
(z_1,\ldots,z_N) = (R_*e^{i(\Omega t+\psi_1)},\ldots,R_*e^{i(\Omega t+\psi_N)}),
\label{eqTorus}
\end{equation}
parametrized by the phases $(\psi_1,\ldots\psi_N)\in\T^N$. This invariant torus is foliated by neutrally stable periodic orbits with period $2\pi/\Omega$ and so for each $0<\lambda<\lambda_0$, the torus is normally hyperbolic. By Fenichel's theorem \cite{Fenichel1979} there is an $\epsilon_0$ (depending on $\lambda$) such that for $0<\epsilon<\epsilon_0$ the invariant torus persists and is $C^r$-smooth for arbitrarily large $r$. Note that reducing $r$ will restrict the $\epsilon_0$: we will need $r\geq 5$ for the approximation to be valid.

We now aim to find the approximating system (\ref{principalEQ2}) on this invariant torus for $0<\epsilon<\epsilon_0$.  We follow a method similar to \cite[Section 7.3]{Wiggins}, using coordinate changes and a slow time to ``blow up'' the weak hyperbolic dynamics. Including all terms up to cubic order (except for the linear term $a_1z_1$ which can be assumed to be contained in $U(z_1)$ by a suitable change in parameters), using Theorem~\ref{thm:expand} we have
\begin{equation}\label{eq:inv5F1}
\begin{split}
F_1 = &  \left[a_{-1}\frac{1}{N} \sum_{j} z_j+
a_2 \frac{z_{1}^{2}}{N} \sum_{j}{{\overline z}_{j}} +  a_3\frac{|z_{1}|^{2}}{N}  \sum_{j}{z_{j}} \right.  \\
& +a_4 \frac{z_{1}}{N} \sum_{j}{|z_{j}|^{2}} + a_5 \frac{z_{1}}{N^2}\sum_{j,k} z_{j}{\overline z}_{k} +  a_6 \frac{{\overline z}_{1}}{N} \sum_{j}  z_{j}^{2} \\
& +a_7 \frac{{\overline z}_{1}}{N^2}\sum_{j,k} z_{j} z_{k}+ a_8 \frac{1}{N} \sum_{j}{|z_{j}|^{2}z_{j}}+ a_9 \frac{1}{N^2}\sum_{j,k} z_{j}^{2}{\overline z}_{k}  \\
& +\left. a_{10} \frac{1}{N^2}\sum_{j,k} z_{j}|z_{k}|^{2}+
a_{11} \frac{1}{N^3} \sum_{j,k,\ell} z_{j} z_{k} {\overline z}_{\ell} \right] + \tilde{F}_1+O(\epsilon).
\end{split}
\end{equation}
where the $\epsilon=0$ error term is $\tilde{F}_1=O(|z|^5)$. The complex normal form coefficients $a_k$ can be expressed using real quantities $\alpha_k$ and $\theta_k$ (or $a_{kR}$ and $a_{kI}$) such that
$$
a_k=\alpha_k e^{i \theta_k}=a_{kR}+i a_{kI},
$$
for $i=-1,1,\ldots,11$. We seek solutions of the following form:
$$
z_k(t)= R_k(t) e^{i(\Omega t + \psi_k (t))}= [R_*+ \rho_k(t)] e^{i(\Omega t + \psi_k (t))} 
$$
that remain close to periodic orbits on the invariant torus (\ref{eqTorus}). In particular, we seek solutions such that $\rho_k$ is small and $\psi_k$ varies slowly with $t$. Re-writing \eqref{eq:inv5}, we have
\begin{equation}
\label{eq:PrincNewCoord1}
\FD{}{t}  {\rho}_1 + i R_1 \left[ \Omega+\FD{}{t} {\psi}_1\right]
= U(R_1) + \epsilon F_1(z_1,\cdots,z_N,0) e^{-i(\Omega t+\psi_1)}+O(\epsilon^2).
\end{equation}
Writing $U$ in real and imaginary parts and expanding for small $\rho_1$ we have
\begin{eqnarray*}
U(R_1) &=& U_R(R_1)+ i R_1 V_I(R_1)\\
&=& U_R(R_*+ \rho_1)+ i R_1 V_I(R_* + \rho_1)\\
&=& U'_R(R_*)\rho_1+ i R_1 [V_I(R_*)+ V_I'(R_*)\rho_1]+O(\rho_1^2).
\end{eqnarray*}
If we define
\begin{equation}
A(\lambda):=\frac{U'_R(R_*)}{\lambda},~~B(\lambda):=\frac{V_I'(R_*)}{\lambda^{1/2}},
\end{equation}
then, from (\ref{eqRstar}),
\begin{equation}
U(R_1) = \lambda A(\lambda)\rho_1 +i R_1[\Omega+\lambda^{1/2} B(\lambda)\rho_1]+O(\rho_1^2).
\end{equation}
This implies that (\ref{eq:PrincNewCoord1}) can be expressed as
\begin{eqnarray}
\FD{}{t}  {\rho}_1 + i R_1 \left[ \Omega+\FD{}{t} {\psi}_1\right] & =&\lambda A(\lambda)\rho_1 +i R_1[\Omega+\lambda^{1/2} B(\lambda)\rho_1]\nonumber\\
& & + \epsilon F_1(z_1,\ldots,z_N)e^{-i(\Omega t+\psi_1)} +O(\epsilon^2)\label{eq:PNC2}
\end{eqnarray}
Recalling from (\ref{eqRstar}) that $R_*^2=\lambda/(-a_{1R})+O(\lambda^2)$, $U(R_*)=U_R(R_*)+iV_I(R_*)R_*=(\lambda +a_{1R} R_*^2+\tau(R_*))R_*$, $\tau(z)=O(z^4)$, and $\tau'(z)=O(z^3)$ so we have
\begin{eqnarray}
A(\lambda) &=& \frac{U_R'(R_*)}{\lambda}\nonumber\\
&=& \frac{\lambda+3 a_{1R}R_*^2 +\tau_R'(R_*)R_*+\tau_R(R_*)}{\lambda}\nonumber\\
&=& 1 + \frac{3a_{1R}}{-a_{1R}}+O(\lambda)\nonumber\\
&=& -2+O(\lambda)\label{eq:Alambda}
\end{eqnarray} 
Similarly, we have
\begin{eqnarray}
B(\lambda)&=& \frac{V'_I(R_*)}{\sqrt{\lambda}} = \frac{2 R_* a_{1I}+\tau_I'(R_*)}{\sqrt{\lambda}} \nonumber\\
&=& \frac{2a_{1I}\sqrt{\lambda}}{\sqrt{\lambda a_{1R}}}(1+O(\lambda))\nonumber\\
&=& \frac{2a_{1I}}{\sqrt{-a_{1R}}}+O(\lambda).\label{eq:Blambda}
\end{eqnarray} 
In particular, for $\lambda\rightarrow 0$ there are finite limits
\begin{equation}
A(0)=-2,~~B(0)= \frac{2a_{1I}}{\sqrt{-a_{1R}}}.
\end{equation}
Note that
$$
\begin{array}{rcl}
a_{-1} h_{-1}(z) e^{-i(\Omega t + \psi_1)} & = & \left[a_{-1} \frac{1}{N} \sum_{j} z_j \right] e^{-i(\Omega t + \psi_1)}\\ 
& = & \alpha_{-1} e^{i\theta_{-1}} \frac{1}{N}\sum_j R_j e^{i(\Omega t + \psi_j)}e^{-i(\Omega t + \psi_1)}\\
& = & \alpha_{-1} \frac{1}{N}\sum_j R_j e^{i(\theta_{-1} + \psi_j- \psi_1)}\\
& = & \alpha_{-1} \frac{1}{N}\sum_j R_j [\cos(\theta_{-1} + \psi_j- \psi_1)+\sin(\theta_{-1} + \psi_j- \psi_1)].
\end{array}
$$
Applying similar expansions for the remaining terms in $F_1$ and taking real parts of (\ref{eq:PNC2}) gives 
\begin{equation}\label{eqR1}
\begin{array}{ll}
\FD{}{t} {\rho}_1(t) = & \lambda A(\lambda) \rho_1 + \epsilon \left[ \alpha_{-1} \sum'_{j} R_j \cos (\theta_{-1}+\psi_j-\psi_1)\right.\\
& +\alpha_2 \sum'_{j} R_1^2 R_j \cos (\theta_2+ \psi_1-\psi_j) \\
& +\alpha_3 \sum'_{j} R_1^2 R_j \cos  (\theta_3+ \psi_j-\psi_1) \\
& +\alpha_4 \sum'_{j} R_1 R_j^2 \cos \theta_4  \\
& +\alpha_5 \sum'_{j,k} R_1 R_j R_k  \cos (\theta_5+ \psi_j-\psi_k)  \\
& +\alpha_6 \sum'_{j} R_1 R_j^2  \cos (\theta_6+ 2 \psi_j- 2\psi_1) \\
& +\alpha_7 \sum'_{i,j} R_1 R_i R_j  \cos [\theta_7+ (\psi_i- \psi_1)+(\psi_j-  \psi_1)] \\
& +\alpha_8 \sum'_{j} R_j^3 \cos (\theta_8+ \psi_j - \psi_1) +  \\
& +\alpha_9 \sum'_{j,k} R_j^2 R_k  \cos (\theta_9+ 2 \psi_j-\psi_k - \psi_1) \\
& +\alpha_{10} \sum'_{j,k} R_j R_k^2  \cos (\theta_{10}+\psi_j - \psi_1)    \\
& +\left.\alpha_{11} \sum'_{i,j,k} R_i R_j R_k  \cos (\theta_{11}+\psi_i+\psi_j-\psi_k- \psi_1) \right]\\
&+O(\rho^2,\epsilon^2)
\end{array}
\end{equation}
where $\rho^2=\max_j(\rho_j^2)$ and $\sum'_{j} a_j:=\frac{1}{N}\sum_{j=1}^N a_j$, $\sum'_{j,k}a_{j,k}:=\frac{1}{N^2} \sum_{j,k=1}^{N}a_{j,k}$, etc are the normalized sums. The equivalent equation for $\psi_1$ is obtained by taking imaginary parts of (\ref{eq:PrincNewCoord1}):
\begin{equation}\label{eqp1}
\begin{array}{ll}
R_1[\Omega+\FD{}{t} {\psi}_1(t)] =  & R_1\Omega+R_1\lambda^{1/2} B(\lambda) \rho_1 + \epsilon \big[ \alpha_{-1} \sum'_{j} R_j \sin(\theta_{-1}+\psi_1-\psi_j)\\
& +\alpha_2 \sum'_{j} R_1^2 R_j \sin (\theta_2+ \psi_1-\psi_j)   \\
& +\alpha_3 \sum'_{j} R_1^2 R_j \sin  (\theta_3+ \psi_j-\psi_1)\\
& +\alpha_4 \sum'_{j}  R_1 R_j^2 \sin \theta_4  \\
& +\alpha_5 \sum'_{j,k}  R_1 R_j R_k  \sin (\theta_5+ \psi_j-\psi_k)   \\
& +\alpha_6 \sum'_{j}  R_1 R_j^2  \sin (\theta_6+ 2 (\psi_j- \psi_1))  \\
&	+\alpha_7 \sum'_{i,j}  R_1 R_i R_j  \sin [\theta_7+ (\psi_i- \psi_1)+(\psi_j-  \psi_1)]  \\
& +\alpha_8 \sum'_{j} R_j^3 \sin (\theta_8+ \psi_j - \psi_1)  \\
& +\alpha_9 \sum'_{j,k} R_j^2 R_k  \sin (\theta_9+ 2 \psi_j-\psi_k - \psi_1)  \\
& +\alpha_{10} \sum'_{j,k} R_j R_k^2  \sin (\theta_{10}+\psi_j - \psi_1)     \\
& +\alpha_{11} \sum'_{i,j,k} R_i R_j R_k  \sin (\theta_{11}+\psi_i+\psi_j-\psi_k- \psi_1)  \big] \\
& + O(\rho^2,\epsilon^2)
\end{array}
\end{equation}
which, after cancellation and dividing through by $R_1$, gives
\begin{equation}\label{eqp1}
\begin{array}{ll}
\FD{}{t} {\psi}_1(t) =  & \lambda^{1/2}B(\lambda) \rho_1 + \epsilon \big[ \alpha_{-1} \sum'_{j} (R_j/R_1) \sin(\theta_{-1}+\psi_1-\psi_j)\\
& +\alpha_2 \sum'_{j} R_1 R_j \sin (\theta_2+ \psi_1-\psi_j)   \\
& +\alpha_3 \sum'_{j} R_1 R_j \sin  (\theta_3+ \psi_j-\psi_1)\\
& +\alpha_4 \sum'_{j}  R_j^2 \sin \theta_4  \\
& +\alpha_5 \sum'_{j,k}  R_j R_k  \sin (\theta_5+ \psi_j-\psi_k)   \\
& +\alpha_6 \sum'_{j}  R_j^2  \sin (\theta_6+ 2 (\psi_j- \psi_1))  \\
&	+\alpha_7 \sum'_{i,j}  R_i R_j  \sin [\theta_7+ (\psi_i- \psi_1)+(\psi_j-  \psi_1)]  \\
& +\alpha_8 \sum'_{j} (R_j^3/R_1) \sin (\theta_8+ \psi_j - \psi_1)  \\
& +\alpha_9 \sum'_{j,k} (R_j^2 R_k/R_1)  \sin (\theta_9+ 2 \psi_j-\psi_k - \psi_1)  \\
& +\alpha_{10} \sum'_{j,k} (R_j R_k^2/R_1)  \sin (\theta_{10}+\psi_j - \psi_1)     \\
& +\alpha_{11} \sum'_{i,j,k} (R_i R_j R_k/R_1)  \sin (\theta_{11}+\psi_i+\psi_j-\psi_k- \psi_1)  \big] \\
& + \frac{1}{R_1} O(\rho^2,\epsilon^2)
\end{array}
\end{equation}

We now define scaled radial variables $r_k$ and a slow time $T$ by
\begin{equation}
\rho_k=\epsilon \frac{R_*(\lambda)}{\lambda} r_k,~~T=\lambda t.
\end{equation} 
For fixed $\lambda$, note that $\rho_k=O(\epsilon)$ and so (\ref{eqR1}) can be written
\begin{equation}\label{eqR2}	
 \FD{}{T} r_1(t) =  A(\lambda) r_1 + f_1 +O(\epsilon),
\end{equation}
where
$$
f_1:= f_1^0 + R^2_*(\lambda) f_1^1 + O(\lambda^2)
$$
and
\begin{equation}
\begin{array}{ll}
f_1^0:=  & \alpha_{-1} \sum'_{j} \cos (\theta_{-1}+\psi_1-\psi_j),\\
f_1^1: =& \alpha_2 \sum'_{j} \cos (\theta_2+ \psi_1-\psi_j) \\
& +\alpha_3 \sum'_{j}  \cos  (\theta_3+ \psi_j-\psi_1) \\
& +\alpha_4 \sum'_{j} \cos \theta_4  \\
& +\alpha_5 \sum'_{j,k} \cos (\theta_5+ \psi_j-\psi_k)  \\
& +\alpha_6 \sum'_{j} \cos (\theta_6+ 2(\psi_j- \psi_1)) \\
& +\alpha_7 \sum'_{i,j} \cos [\theta_7+ (\psi_i- \psi_1)+(\psi_j-  \psi_1)] \\
& +\alpha_8 \sum'_{j} \cos (\theta_8+ \psi_j - \psi_1) +  \\
& +\alpha_9 \sum'_{j,k} \cos (\theta_9+ 2 \psi_j-\psi_k - \psi_1) \\
& +\alpha_{10} \sum'_{j,k} \cos (\theta_{10}+\psi_j - \psi_1)    \\
& +\alpha_{11} \sum'_{i,j,k} \cos (\theta_{11}+\psi_i+\psi_j-\psi_k- \psi_1).
\end{array}
\end{equation}
Similarly, (\ref{eqp1}) can be written
\begin{equation}
\label{eqp2}
\FD{}{T} {\psi}_1(t) =  \epsilon\lambda^{-1} C(\lambda) r_1 + \epsilon\lambda^{-1} h_1 + O(\epsilon^2)
\end{equation}
where
\begin{eqnarray*}
C(\lambda):&=&\frac{R_*(\lambda)B(\lambda)}{\sqrt{\lambda}} = -2\frac{a_{1I}}{a_{1R}}+O(\lambda),\\
h_1:&=&h_1^0+ R^2_*(\lambda) h_1^1+O(\lambda^2)
\end{eqnarray*}
and
\begin{equation}
\begin{array}{ll}
h_1^0 :=  & \alpha_{-1} \sum'_{j} \sin(\theta_{-1}+\psi_1-\psi_j),\\
h_1^1 := & \alpha_2 \sum'_{j} \sin (\theta_2+ \psi_1-\psi_j)   \\
& +\alpha_3 \sum'_{j} \sin  (\theta_3+ \psi_j-\psi_1)\\
& +\alpha_4 \sum'_{j} \sin \theta_4  \\
& +\alpha_5 \sum'_{j,k} \sin (\theta_5+ \psi_j-\psi_k)   \\
& +\alpha_6 \sum'_{j} \sin (\theta_6+ 2 (\psi_j- \psi_1))  \\
&	+\alpha_7 \sum'_{i,j} \sin [\theta_7+ (\psi_i- \psi_1)+(\psi_j-  \psi_1)]  \\
& +\alpha_8 \sum'_{j} \sin (\theta_8+ \psi_j - \psi_1)  \\
& +\alpha_9 \sum'_{j,k} \sin (\theta_9+ 2 \psi_j-\psi_k - \psi_1)  \\
& +\alpha_{10} \sum'_{j,k} \sin (\theta_{10}+\psi_j - \psi_1)     \\
& +\alpha_{11} \sum'_{i,j,k} \sin (\theta_{11}+\psi_i+\psi_j-\psi_k- \psi_1)
\end{array}
\end{equation}

In summary, we can write system (\ref{eqR2},\ref{eqp2}) as
\begin{equation}
\begin{array}{rcl}
\frac{d}{dT} r_j & = & A(\lambda) r_j +  f_j + O(\epsilon) \\
\frac{d}{dT} \psi_j & = & \epsilon\lambda^{-1} \left[C(\lambda) r_j+  h_j\right] + O(\epsilon^2)
\end{array}
\label{eqrescaled}
\end{equation}
for $j=1,\ldots, N$. Note that $A$, $C$, $f_j$ and $h_j$ have finite limits as $\lambda\rightarrow 0$ and so (\ref{eqrescaled}) gives a slow timescale for evolution of $\psi_j$ as long as
$$
\epsilon =o(\lambda)
$$
which holds, for example, if $\epsilon= \lambda^2$.
Defining a new set of amplitude variables
$$
\sigma_j := r_j + \frac{f_j(\psi_1,\ldots,\psi_{N-1})}{A(\lambda)},
$$
system \eqref{eqrescaled} for fixed $\lambda$ becomes 
\begin{equation}
\begin{array}{rcl}
\frac{d}{dT} \sigma_j & = & A(\lambda) \sigma_j + O(\epsilon) \\
\frac{d}{dT} \psi_j & = & \epsilon \lambda^{-1} \left[ C(\lambda) \left[\sigma_j-\frac{f_j}{A(\lambda)}\right] +  h_j\right]+O(\epsilon^2).
\end{array}
\label{eqrescaled2}
\end{equation}
which gives
\begin{equation}
\begin{array}{rcl}
\frac{d}{dT} \sigma_j & = & A(\lambda) \sigma_j + O(\epsilon) \\
\frac{d}{dT} \psi_j & = & \epsilon\lambda^{-1} \left[ C(\lambda) \sigma_j+H_j\right] + O(\epsilon^2), 
\end{array}
\label{eqrescaled3}
\end{equation}
where 
\begin{equation}
H_j=h_j - \frac{C(\lambda)}{A(\lambda)} f_j.
\label{eq:Hj}
\end{equation}
Let us write
\begin{equation}
H_j= H_j^0+\lambda H_j^1+O(\lambda^2),
\end{equation} 
where 
\begin{equation}
\begin{split}
H_j^0 =& h^0_j - \frac{C(0)}{A(0)} f^0_j,\\
H_j^1 =& \frac{R_*^2(\lambda)}{\lambda} \left[ h^1_j - \frac{C(0)}{A(0)} f^1_j\right] - \frac{C'(0)A(0)-A'(0)C(0)}{A(0)^2} f^0_j,
\end{split}
\end{equation}
and these are trigonometric polynomials in $\psi_k-\phi_j$ such that $H_j^0$ only involves pairwise coupling (and on $\alpha_{-1}$ while $H_j^1$ includes coupling of up to four phases (and on $\alpha_{2},\ldots,\alpha_{11}$).

Applying Fenichel's theorem \cite[Theorem 9.1]{Fenichel1979} to  \eqref{eqrescaled3}, for all $0<\lambda<\lambda_0$ there is a $0<\epsilon_0=o(\lambda)$ such that whenever $|\epsilon|<\epsilon_0$ the evolution of the phases is given by
$$
\FD{}{T} \psi_j = \epsilon\lambda^{-1} H_j+ O(\epsilon^2) = \epsilon\lambda^{-1} [H_j^0+\lambda H_j^1]+O(\epsilon\lambda).
$$
The solutions of this reduced equation are approximated by solutions of
$$
\FD{}{T} \psi_j = \epsilon\lambda^{-1} H_j = \epsilon\lambda^{-1} [H_j^0+\lambda H_j^1]
$$
over an interval of time $0<T<\tilde{T}$ with $\tilde{T} = O(\epsilon^{-1}\lambda^{-1})$.  In terms of the original phases $\phi_j$ and time $t$, this reduced equation is
\begin{equation}
\FD{}{t} \phi_j = \Omega+\epsilon [H_j^0+\lambda H_j^1]
\label{eq:dphiH}
\end{equation}
where the phase differences $\psi_j-\psi_k=\phi_j-\phi_k$ for all $j$ and $k$, and the approximation will be close for times $0<t<\tilde{t}$ with $\tilde{t}=O(\epsilon^{-1}\lambda^{-2})$.

For $k=-1,1,\ldots,11$ we define $\beta_k$ and $\gamma_k$ such that for all $\theta$
$$
\beta_k \cos(\gamma_j+\theta) := \alpha_k \sin(\theta_k+\theta) -\frac{C(0)}{A(0)}\alpha_k \cos (\theta_k+\theta).
$$
Then we can write (\ref{eq:dphiH}) in the form 
\begin{equation}\label{principalEQ2trun}
\begin{split}
\FD{}{t} {\phi}_j = &  \Omega+\epsilon H_1\\
= & \tilde{\Omega}(\phi,\epsilon) + \frac{\epsilon}{N}\sum_{k=1}^N g_2(\phi_k-\phi_j)+\frac{\epsilon}{N^2}\sum_{k,\ell=1}^{N}g_3(\phi_k+\phi_{\ell}-2\phi_j) \\ 
& + \frac{\epsilon}{N^2}\sum_{k,\ell=1}^{N} g_4(2\phi_k-\phi_{\ell}-\phi_j) + \frac{\epsilon}{N^3}\sum_{k,\ell,m=1}^{N} g_5(\phi_k+\phi_{\ell}-\phi_{m}-\phi_j)
\end{split}
\end{equation}
which is equivalent to a truncation of \eqref{principalEQ2} such that the coupling is as follows:
\begin{equation}
\begin{split}
\tilde{\Omega}(\phi,\epsilon) =& \Omega + R_*^2\eps \left[\beta_4 \cos \gamma_4+  \frac{\beta_5}{N^2}\sum_{j,k}  \cos (\gamma_5+ \phi_j-\phi_k) \right] \\
g_2(\varphi) = & \beta_{-1} \cos(\gamma_{-1}+\varphi) + R_*^2  \left[ 
\beta_2 \cos (\gamma_2 - \varphi) 
+ \beta_3 \cos (\gamma_3 + \varphi)\right.\\  
& \left. + \beta_6 \cos (\gamma_6 + 2\varphi)  
+ \beta_8 \cos (\gamma_8 + \varphi)  
+ \beta_{10} \cos (\gamma_{10} + \varphi)  
\right]\\
& - \lambda \frac{C'(0)A(0)-A'(0)C(0)}{A(0)^2} \alpha_{-1} \cos(\theta_{-1}+\varphi)
\\
g_3(\varphi) =& R_*^2 \left[\beta_7 \cos (\gamma_7+ \varphi)\right]
\\
g_4(\varphi) =& R_*^2 \left[\beta_9 \cos (\gamma_9+\varphi)\right]
\\
g_5(\varphi) =& R_*^2 \left[\beta_{11} \cos (\gamma_{11}+\varphi)\right].
\end{split}
\label{eqcoups}
\end{equation}

Substituting in the leading order of the expansion $R_*^2= \lambda/(-a_{1R})+O(\lambda^2)$ and simplifying the trigonometric expressions, we obtain the leading order terms in a phase description of the bifurcating solutions, involving the terms as expressed in \eqref{eq:coupform2}. Note that the $\alpha_{-1}$ term in $g_2$ involves (via $A'(0)$ and $C'(0)$) coefficients at fifth order in the normal form for (\ref{eqhopf}).

\section{Examples and consequences}
\label{sec:examples}

\subsection{A numerical example}

We briefly give a numerical example that illustrates the reduction in Theorem~\ref{thm:main}. Consider the system of $N$ globally coupled Stuart-Landau oscillators of the form \eqref{eq:inv5} where we choose all parameters zero except for
\begin{equation}
\label{eq:parameg}
\lambda=0.1,~\omega=1,~\eps=0.5,~a_{1r}=-1,~a_{2r}=0.3
\end{equation}
Figure~\ref{fig:examplets} shows some time series for the system \eqref{eq:inv5} with these parameters along with the case $a_{2r}=-0.3$; (a) shows a stable antiphase solution while (b) shows a stable in-phase solution.

Using the results above we re-write in terms of the phase only equations. In case (a) we have $R_*=\sqrt{0.1}=0.3162$, $\alpha_2=0.1$, $\theta_2=0$, $\beta_2=0.2$, $\gamma_2=-\pi/2$; all of the $g_k$ are zero except for $g_2(\phi)= 0.3\sin (\phi)$. In case (b) we similarly get $g_2(\phi)= -0.3\sin(\phi)$. Figure~\ref{fig:examplereducedts} presents two time series for the system \eqref{principalEQ}). Observe the qualitatively (and indeed quantitatively) similar behaviour of \eqref{eq:inv5} and the reduced system \eqref{principalEQ} in this case, even though $\lambda$ and $\epsilon$ are comparatively large.

\subsection{Non-pairwise coupling and synchrony}

We discuss certain periodic cluster states, and find that the presence of the additional ``non-pairwise coupling'' terms in \eqref{principalEQ} implies a wider range of behaviour that is possible for pairwise coupling \eqref{eq:TNkuramoto}. Note that  particular isotropy subgroups for $S_N \times \T^1$ symmetry that generally contain periodic solutions \cite{AshSwi91}. Using the notation $\zeta= e^{2 \pi i / N}$, we recall from \cite{AshSwi91} that in particular the following fixed point subspaces are invariant:
$$
\begin{array}{ll}
{\rm Fix}(S_N)=&\{ (z,\ldots,z ): z \in \C \},\\
{\rm Fix}(\Z_N)=&\{ (z,z \zeta,z \zeta^2,\ldots, z \zeta^{-1}): z \in \C \}
\end{array}
$$
and indeed for any factorization $N=km$ the following fixed point subspace is also invariant:
$$
{\rm Fix}((S_k)^m \otimes \Z_m)=\{ (z,\ldots,z,z \zeta^k,\ldots,z \zeta^{-k} ): z \in \C \}
$$
The above fixed point spaces all contain a periodic orbit after the Hopf bifurcation ($\lambda>0$), for sufficiently small $\epsilon$, even on inclusion of non-pairwise coupling terms. This can be verified by considering the reduced equations \eqref{principalEQ} and noting that within each of these invariant spaces we can reduce to a phase equation $\dot{\Phi}$ is a constant that depends on the coupling: for example, for full synchrony we have
$$
{\rm Fix}(S_N)=\{ (\Phi,\ldots,\Phi ) \}
$$
and so, using (\ref{principalEQ2trun},\ref{eqcoups}) we can determine that the frequency of the fully synchronous state is
$$
\dot{\Phi} = \Omega+ \epsilon \beta_{-1}\cos \gamma_{-1}+ \epsilon R_*^2 \sum_{j=2}^{11} \beta_j\cos(\gamma_j)-\epsilon \lambda \frac{C'(0)A(0)-A'(0)C(0)}{A(0)^2}\alpha_{-1}\cos(\theta_{-1}).
$$
In principle also the stability of the synchronous state can be determined, as can the frequencies and stabilities of the other periodic orbits.

\subsection{Non-pairwise coupling and two-cluster states}
\label{sec:twocluster}

We study two cluster states in more detail: these have isotropy subgroup $S_Q \times S_P$ with $Q+P=N$. In this case we have
$$
{\rm Fix}(S_Q \times S_P)=\{ (\underbrace{\phi_1,\ldots,\phi_1}_Q,\underbrace{\phi_2,\ldots,\phi_2}_P): \phi_1,\phi_2 \in \R  \}.
$$
For simplicity we assume here that $\tau(z):=0$ in the uncoupled system (\ref{eqhopf}) so that $A'(0)=C'(0)=0$. Restricting \eqref{principalEQ2trun} to the two-cluster subspace and suppressing the $O(\epsilon^2)$ terms we get
\begin{equation*}
\begin{split}
\dot \phi_1 &= \Omega + \epsilon H_1(\phi_1,\phi_2,Q,P) \\
\dot \phi_2 &= \Omega + \epsilon H_2(\phi_1,\phi_2,Q,P)
\end{split}
\end{equation*}
where $H_1$ and $H_2$ are given in Appendix~\ref{app:twocluster}. If we write the phase difference between the clusters as $\Psi := \phi_1-\phi_2$ then
\begin{equation}
\dot \Psi = \epsilon[H_1(\phi_1,\phi_2,Q,P)-H_2 (\phi_1,\phi_2,Q,P)].
\label{eqPsidot}
\end{equation}
This ODE can be written
\begin{equation}
\begin{split}
\dot \Psi = \epsilon G(\Psi):&
= 2 \epsilon \sin \frac{\Psi}{2} \left [ A_1 \cos \frac{\Psi}{2} + B_1 \sin \frac{\Psi}{2} + A_2 \cos \frac{3\Psi}{2} + B_2 \sin \frac{3\Psi}{2}\right].
\end{split}
\label{eq:Psidash}
\end{equation}
The coefficients $A_i,B_i$, $i=1,2$ can be expressed in terms of $\alpha=(Q-P)/N$ with $\alpha\in(-1,1)$ and the $\beta_j$, $\gamma_j$ for $j=2,3,6,\ldots,11$ according to the following:
\begin{equation}
\begin{split}
A_1 & = R_*^2\left(\frac{3+\alpha^2}{4} \beta_{11} \sin\gamma_{11} + \frac{3-\alpha^2}{4} \beta_{7} \sin\gamma_{7} +\frac{1+\alpha^2}{2} \beta_{9}  \sin\gamma_{9}\right.\\
& \left. -\beta_2\sin \gamma_2+\sum_{j=3,6,8,10} \beta_{j}  \sin\gamma_{j}\right) +\beta_{-1}\sin \gamma_{-1}\\
B_1 & =  R_*^2\left(\frac{\alpha+3\alpha^3}{4} \beta_{11} \cos\gamma_{11}+ \alpha\sum_{j=2,3,6,7,8,9,10} \beta_{j} \cos\gamma_{j}\right)+\alpha \beta_{-1}\cos \gamma_{-1}\\
A_2 & =  R_*^2\left(\frac{1-\alpha^2}{4} \beta_{11} \sin\gamma_{11}+ \frac{1+\alpha^2}{2} \beta_{7} \sin\gamma_{7} + \beta_{6} \sin\gamma_{6}\right)\\
B_2 & =  R_*^2\left(\frac{\alpha-\alpha^3}{4} \beta_{11} \cos\gamma_{11}+ \alpha (\beta_{6} \cos\gamma_{6}+ \beta_{7} \cos\gamma_{7})\right)
\end{split}
\label{eq:AB12}
\end{equation}
where for details we refer to Appendix~\ref{app:twocluster}. These expressions allow us to draw some conclusions about two-cluster states and, in particular, the influence of non-pairwise interactions.
\begin{itemize}
\item Firstly, we note that the factor of $\sin (\Psi/2)$ corresponds to there always being a solution $G(0)=0$: this corresponds the fully synchronous solution with symmetry $S_N$.
\item Secondly, in the special case $N$ even and $P=Q$ (so that $\alpha=0$), $B_1=B_2=0$ meaning that there is also a solution $G(\pi)=0$: this corresponds to a solution with symmetry $(S_P)^2\otimes \Z_2$.
\item Thirdly, loss of linear stability of synchrony is associated with change in sign of $G'(0)$. More precisely, synchrony is stable if $A_1+A_2<0$ and unstable if $A_1+A_2>0$.
\item Fourthly, if there is a root $G(\Psi)$ with $\Psi\neq 0$ mod $2\pi$ this is a non-trivial two cluster state. This bifurcates from the fully synchronous solution where $G'(0)=0$. The only possibility of there being NO non-trivial two-cluster solution is at such a bifurcation point, such that $G'(0)=0$ and $G(\Psi)$ has the same sign for all $\Psi\in(0,2\pi)$.
\end{itemize}

Moreover, the dependence on $\alpha$ of $A_i$ and $B_i$ means we can conclude the following result about the set of possible two-cluster states. Note that the case
\begin{equation}
\beta_{7}=\beta_{9}=\beta_{11}=0
\label{eq:pairwisebeta}
\end{equation}
corresponds to the case of there being only pairwise coupling in \eqref{principalEQ}.

\begin{theorem}
\label{thm:twoclusters}
For system \eqref{eq:Psidash} in the special case of pairwise coupling \eqref{eq:pairwisebeta} and any $\Psi_0\in(0,2\pi)$ either
\begin{itemize}
\item $G$ is independent of $\alpha$ and $G(\Psi_0)=0$ for all $\alpha\in (-1,1)$, or
\item $G(\Psi_0)=0$ for at most one $\alpha\in(-1,1)$.
\end{itemize}
In the more general case where \eqref{eq:pairwisebeta} does not hold, there can be an additional case
\begin{itemize}
\item $\Psi_0$ is a root of $G(\Psi)=0$ for two distinct $\alpha\in(-1,1)$.
\end{itemize}
Moreover, this case does appear for certain choices of parameters.
\end{theorem}

\proof
This follows on noting that if there is pairwise coupling then the dependence on $\alpha$ is $A_1=a_1$, $B_1=\alpha b_1$, $A_2=a_2$, $B_2=\alpha b_2$. This means that non-trivial roots $\Psi_0$ of $G(\Psi)=0$, from \eqref{eq:AB12} must satisfy
$$
\alpha \left(b_1\sin \frac{\Psi}{2}+b_2 \sin \frac{3\Psi}{2}\right)=-a_1\cos \frac{\Psi}{2}+a_2 \cos \frac{3\Psi}{2}
$$
and so either $B_1=B_2=0$ for all $\alpha$ and $\Psi_0$ is a root of the right hand side, or there is precisely one real $\alpha$ that satisfies this equation. If this quantity satisfies $-1<\alpha<1$ then there is a corresponding two cluster state with close to this phase difference, for large enough $N$ (exactly this phase difference if $\alpha$ is rational).

On the other hand, if \eqref{eq:pairwisebeta} is not satisfied there may be quadratic dependence of $G$ on $\alpha$ leading to the possibility of two $\alpha$. For example, consider the specific case where the normal form coefficients are such that $A_1=1/8+ \alpha^2$, $B_1=-3\alpha/4$ and $A_2=B_2=0$. This has nontrivial roots $\Psi$ where
$$
-\frac{3}{4} \alpha \sin \frac{\Psi}{2} + \left(\frac{1}{8}+ \alpha^2\right) \cos \frac{\Psi}{2},
$$
namely there are roots at $\Psi=\pi/2$ when $\alpha^2 -3\alpha/4+1/8=0$. This implies there are two cluster states with this phase difference for the two isolated values
$$
\alpha\in\{1/2,1/4\}.
$$
\qed

~

Theorem~\ref{thm:twoclusters} highlights a particular restriction in the existence of two-cluster states that holds for very general systems of the form \eqref{eq:TNkuramoto}. Only on addition of additional interaction terms of the form shown in \eqref{principalEQ},\eqref{eq:coupform} do we start to find the sort of behaviour one would expect of a generic symmetric system on $\T^N$.

\section{Discussion}
\label{sec:discussion}

In summary, we study a system of $N$ identical systems near  generic Hopf bifurcation that are symmetrically and weak coupled. In such a case, reduction to a phase description will be possible for some neighbourhoor, but a pairwise-coupling models such as \eqref{eq:TNkuramoto} may miss a number of qualitatively different terms. 

Note (\ref{principalEQ},\ref{eq:coupform}) can be approximated by (\ref{eq:TNkuramoto},\ref{eq:kuramotocoupling}) but this approximation and the true solution may move apart over a timescale of order $O(\epsilon^{-1}\lambda^{-1})$. The next approximation includes two, three and four-phase interactions, from cubic nonlinearities in the equations and this will be valid for the longer timescale $O(\epsilon^{-1}\lambda^{-2})$. These terms will be important especially near a secondary bifurcation where well-known degeneracies of (\ref{eq:TNkuramoto},\ref{eq:kuramotocoupling}) will be unfolded. 

We cannot guarantee that the truncation (\ref{principalEQ2trun}) has the same qualitative dynamics as (\ref{principalEQ2}) unless the dynamics of the former is robust to addition of higher order terms. Although $\lambda$ is small, we assume $\epsilon$ must be smaller in order for the reduction to hold - in the event that this does not hold then there may be solutions where the amplitude of the different oscillators may vary considerably, and the picture of possible dynamics may be much richer \cite{DR05}. In addition to the restrictions on the number of phases interacting, the form of the functions is illuminating - \eqref{eqcoups} reveals that while $g_2$ has first and second harmonics, $g_3$-$g_5$ have only first harmonics to this lowest order. 

The presence of the phase-dependent frequency detuning $\tilde{\Omega}(\phi,\epsilon)$ in \eqref{principalEQ} is somewhat surprising. This term is invariant under permutations of the arguments of $\phi$ and so does not give any effect in the phase difference dynamics and hence on the synchrony properties of the system. However, it will be a measurable effect, for example, affecting the frequencies of different states of synchrony that will be missing from the system (\ref{eq:TNkuramoto}).

If we look at the special case of {\em pairwise} (or the even more special case of {\em linear}) coupling of nonlinear systems undergoing Hopf bifurcation, of the form
$$
\FD{}{t} {z}_1= f(z_1)+ \epsilon \sum_{j=1}^{N}  g(z_1,z_j)
$$
then clearly this will give $g_3$, $g_4$ and $g_5$ identically zero. While the $g_2$ terms in \eqref{principalEQ} may be quite complex and can have higher Fourier modes - previous work on this pairwise coupling has demonstrated present of a rich range of behaviours including robust attracting heteroclinic cycles (slow switching) \cite{AshSwi91} and arbitrary cluster states \cite{Dai96,Oro2009}. We remark that terms present in $g_2$ are also present in the work of Hansel {\em et al} \cite{Han93} and the generalization \cite{AOWT07}.

The example we give in Theorem~\ref{thm:twoclusters} is a new but  subtle dynamical effect that can appear for non-pairwise coupling. We expect there are more remarkable implications of non-pairwise coupling waiting to be discovered. For example, there are systems of pairwise coupled system of the form (\ref{eq:TNkuramoto}) with $N=4$ or more identical oscillators that possess chaotic attracting states, at least for $g(\phi)$ with least four harmonics \cite{Bick-et-al}. It is still unknown whether there is a $g_2$ that gives chaotic attractors for all sufficiently large $N$. We speculate that the additional terms $g_3$-$g_5$ may give broader regions of existence of chaotic attractors for (\ref{principalEQ},\ref{eq:coupform}) using coupling functions with fewer harmonics.

We finish by mentioning a couple of examples from the literature that have non-pairwise coupling. In \cite{RosPik2007,PikRos2009,BurPik11} the authors consider coupled phase oscillators of the form
\begin{equation}\label{KuraR}
\dfrac{d}{dt} {\phi}_1= \omega-\frac{K}{N}\sum_{k=1}^{N} \sin( \phi_1-\phi_k + \alpha R)
\end{equation}
where the order parameter is
$$
R= \left| \frac{1}{N}\sum_{\ell=1}^{N} \exp i\phi_\ell \right|.
$$
Writing  \eqref{KuraR} in terms of phases only we find
\begin{equation}
\label{KuraR2}
\dfrac{d}{dt} {\phi}_1= \omega-\frac{K}{N}\sum_{k=1}^{N} \sin\left( \phi_1-\phi_k + \alpha 
\sqrt{\frac{1}{N^2}\sum_{k,\ell=1}^{N} \exp i(\phi_\ell-\phi_k)}\right)
\end{equation}
which includes non-pairwise terms that include all $N$ phases. This can be derived \cite{RosPik2007} by assuming that the coupling is via a dynamic mean field - note that this justification assumes there is an extra active degree of freedom in the coupling. In this paper we assume the only dynamic variables are those undergoing Hopf bifurcation. Close in spirit to our paper is the analysis of \cite{KomPik2013} who consider three communities of oscillators with frequencies $\omega_i$, $i=1,2,3$ such that $\omega_1+\omega_2\approx \omega_3$. They find three-phase interactions in the phases of the order parameters, using an Ott-Antonsen reduction \cite{OttAnt2008}. By comparison, we are dealing simply with the oscillator phases, and the general form of our equations means that we are unable to apply the method of \cite{OttAnt2008}.

%

\appendix

\newpage

\section{Coefficients for two cluster states}
\label{app:twocluster}

Note that if we write $p=P/N$ and $q=Q/N$ then (\ref{eqcoups}) implies that for (\ref{eqPsidot}) we have
\begin{equation*}
\begin{split}
H_1 (\phi_1,\phi_2,Q,P) = &  \beta_{-1} [q \cos \gamma_{-1}+p \cos (\gamma_{-1}-\phi_1+\phi_2)]\\
&+R_*^2\left[ \beta_2 [q  \cos (\gamma_2)+ p  \cos (\gamma_2+ \phi_1-\phi_2)]\right.  \\
& +\beta_3 [ q  \cos (\gamma_3)+  p  \cos  (\gamma_3+ \phi_2-\phi_1)] +
 \beta_4 (p+q)   \cos \gamma_4 \\
& +\beta_5  [ (p^2+q^2) \cos \gamma_5 + pq \cos (\gamma_5 + (\phi_2-\phi_1))+ pq \cos (\gamma_5 + (\phi_1-\phi_2))] \\
& +\beta_6     [ q \cos \gamma_6 + p \cos (\gamma_6+ 2 (\phi_2- \phi_1))   ] \\
& +\beta_7    [ q^2 \cos (\gamma_7) + 2pq \cos (\gamma_7+ (\phi_2- \phi_1)) + p^2 \cos (\gamma_7+ 2 (\phi_2- \phi_1))  ]   \\
& +\beta_8   (q \cos \gamma_8 + p \cos (\gamma_8+ \phi_2 - \phi_1) )  \\
& +\beta_9  [ q^2 \cos \gamma_9 + pq \cos (\gamma_9+ 2 (\phi_2- \phi_1)) \\
& + pq \cos (\gamma_9- \phi_2 + \phi_1) + p^2  \cos (\gamma_9 +  \phi_2- \phi_1)]  \\
& +\beta_{10}   ( q \cos \gamma_{10} + p \cos (\gamma_{10}+\phi_2 - \phi_1) )  \\
& +\beta_{11}  [q^3 \cos \gamma_{11} + 2 pq^2 \cos(\gamma_{11} + \phi_2 - \phi_1) + qp^2 \cos(\gamma_{11} + 2(\phi_2 - \phi_1))\\
& \left.+pq^2 \cos(\gamma_{11} + \phi_1 - \phi_2)+ 2qp^2 \cos \gamma_{11} +p^3 \cos(\gamma_{11} + \phi_2 - \phi_1)\right].
\end{split}
\end{equation*}
and $H_2 (\phi_1,\phi_2,Q,P) =  H_1 (\phi_2,\phi_1,P,Q)$ and so if we set $\Psi := \phi_1-\phi_2$ then
\begin{equation}
\FD{}{t} \Psi = \epsilon[H_1(\phi_1,\phi_2,Q,P)-H_2 (\phi_1,\phi_2,Q,P)]=:\epsilon G(\Psi).
\end{equation}
This can be written
\begin{equation}
\begin{split}
G(\Psi) = & \beta_{-1}[(q-p)\cos \gamma_{-1}+p\cos (\gamma_{-1}-\Psi)-q\cos (\gamma_{-1}+\Psi)]\\
& +R_*^2\left[\beta_2 [(q-p)  \cos \gamma_2+ p \cos (\gamma_2+ \Psi)- q \cos (\gamma_2 - \Psi)]\right.    \\
& +\beta_3  [(q-p)  \cos \gamma_3+  p  \cos  (\gamma_3 - \Psi)- q \cos  (\gamma_3+\Psi) ] \\
& +\beta_6   [(q-p) \cos \gamma_6 + p \cos (\gamma_6 - 2 \Psi) - q \cos (\gamma_6+ 2 \Psi )   ]  \\
& +\beta_7  [  (q^2-p^2) \cos \gamma_7 + 2pq \cos (\gamma_7 - \Psi) - 2pq \cos (\gamma_7+ \Psi)]\\
& +\beta_7  [ p^2 \cos (\gamma_7 - 2 \Psi) - q^2 \cos (\gamma_7+ 2 \Psi )  ]   \\
& +\beta_8  [(q-p) \cos \gamma_8 + p \cos (\gamma_8 - \Psi) - q \cos (\gamma_8+ \Psi) ]   \\
& +\beta_9   [(q^2-p^2) \cos \gamma_9 + p^2 \cos (\gamma_9 - \Psi) - q^2 \cos (\gamma_9+\Psi ) ] \\
& +\beta_{10} [( q-p) \cos \gamma_{10} + p \cos (\gamma_{10}-\Psi) ) - q \cos (\gamma_{10}+\Psi) )  ]\\
& +\beta_{11} [(q^3- 2pq^2+2qp^2 - p^3)  \cos \gamma_{11}] \\
& +\beta_{11} [(2 pq^2-qp^2+p^3) \cos(\gamma_{11} -\Psi)]\\
& +\beta_{11} [(-q^3+pq^2-2 qp^2) \cos(\gamma_{11} + \Psi)] \\
& \left.+\beta_{11}[ qp^2 \cos(\gamma_{11} -2 \Psi) - pq^2 \cos(\gamma_{11} + 2 \Psi)]\right].
\end{split}
\end{equation}
Now writing $1-\cos \Psi = 2\sin^2 (\Psi/2)$, $\sin \Psi= 2 \sin (\Psi/2)\cos (\Psi/2)$ and $1-\cos 2\Psi = 2\sin^2 \Psi$ we obtain expressions 
\begin{equation}
\begin{split}
G(\Psi):&= A_1 \sin \Psi + B_1 (1-\cos \Psi) \\
 &+ A_2 (\sin 2\Psi-\sin \Psi) + B_2 (\cos \Psi- \cos 2 \Psi) \\
&= 2 \sin \frac{\Psi}{2} \left [ A_1 \cos \frac{\Psi}{2} + B_1 \sin \frac{\Psi}{2} + A_2 \cos \frac{3\Psi}{2} + B_2 \sin \frac{3\Psi}{2}\right]\\
\end{split}
\end{equation}
where $A_i,B_i$, $i=1,2$ depend on $p,q$ and the $\beta_j$, $\gamma_j$ for $j=-1,2,3,6,\ldots,11$ as follows:
\begin{equation}
\begin{split}
A_1 & := \beta_{-1}(p+q)\sin \gamma_{-1} +R_*^2[-\beta_{2}(p+q)\sin \gamma_{2} \\
& + \beta_{11} (p^3 +2 p^2 q +2  p q^2 + q^3)\sin\gamma_{11} + \beta_{7} (p^2 +4 p q + q^2 )\sin\gamma_{7}\\
&+\beta_{9} (p^2 + q^2) \sin\gamma_{9}+ \beta_{10} (p+ q) \sin\gamma_{10} + \beta_{2}( p + q) \sin\gamma_{2}\\
&+ \beta_{3} (p+ q) \sin\gamma_{3}+ \beta_{6}( p+q) \sin\gamma_{6}+ \beta_{8}( p+q) \sin\gamma_{8}], \\
B_1 & :=  \beta_{-1}(q-p)\cos\gamma_{-1}
+R_*^2[\beta_{2}(q-p)\cos \gamma_{2} \\
& + \beta_{11}( -p^3 +2 p^2 q -2 p q^2 +q^3) \cos\gamma_{11}+ \beta_{7}(q^2 -p^2) \cos\gamma_{7}\\
& + \beta_{9}(q^2- p^2) \cos\gamma_{9}+ \beta_{10}(q-p) \cos\gamma_{10}+ \beta_{2}(q- p) \cos\gamma_{2}\\
&+ \beta_{3}(q-p) \cos\gamma_{3}+ \beta_{6}(q- p) \cos\gamma_{6}+ \beta_{8}(q-p) \cos\gamma_{8}], \\
A_2 & :=  R_*^2[\beta_{11}( p^2 q +pq^2) \sin\gamma_{11}+ \beta_{7}( p^2+ q^2) \sin\gamma_{7} + \beta_{6}( p+ q) \sin\gamma_{6}],\\
B_2 & :=  R_*^2[\beta_{11} (pq^2-p^2 q) \cos\gamma_{11}+ \beta_{7}(q^2- p^2) \cos\gamma_{7}+ \beta_{6}(q- p) \cos\gamma_{6}].
\end{split}
\end{equation}
Finally, we define $\alpha=q-p$ with $\alpha\in(-1,1)$ so that $p=(1+\alpha)/2$, $q=(1-\alpha)/2$ which gives the expressions in \eqref{eq:AB12}.

~

\newpage

\begin{figure}%
\includegraphics[width=\columnwidth,clip=]{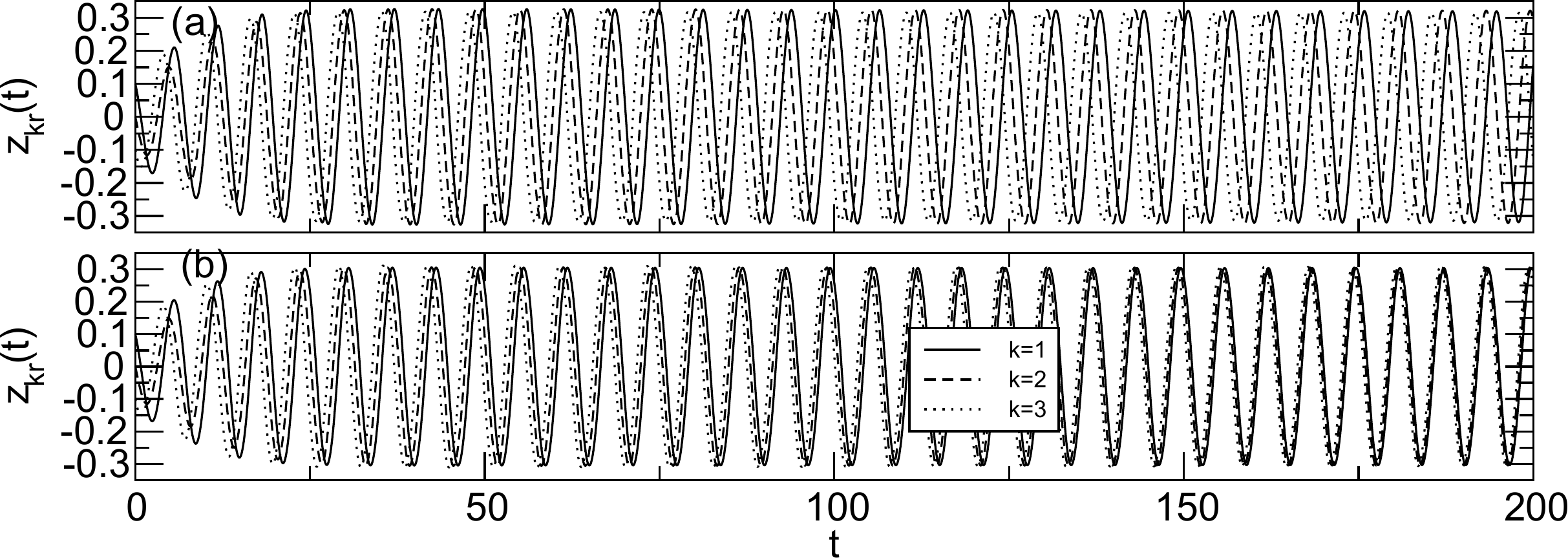}%
\caption{Example time series for coupled oscillators (\ref{eq:inv5}) showing the real parts $z_{kr}(t)$ against $t$ for $N=3$ and two different parameter sets; (a) shows evolution to an anti-phase solution for parameters \eqref{eq:parameg} while (b) shows evolution to an in-phase solution for the same parameters except $a_{2r}=-0.3$.}%
\label{fig:examplets}%
\end{figure}

\begin{figure}%
\includegraphics[width=\columnwidth,clip=]{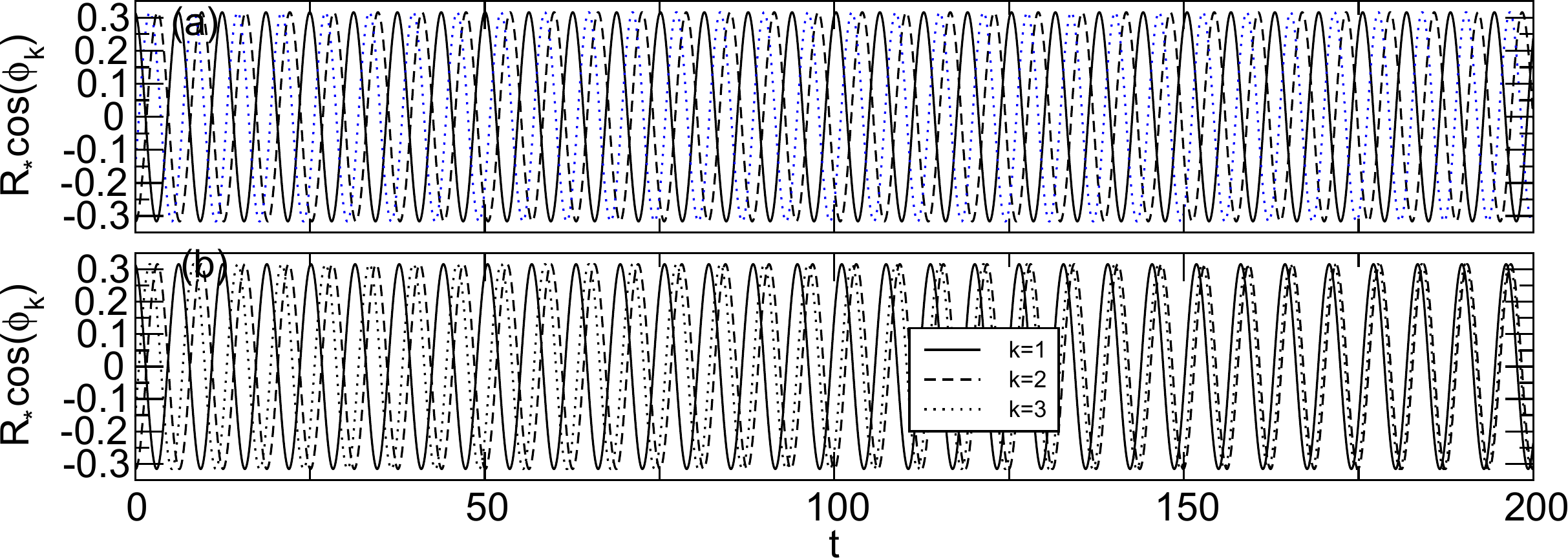}%
\caption{Example time series for a phase approximation (\ref{principalEQ}) of the system in Figure~\ref{fig:examplets}. It shows $R_*\cos(\phi_k(t))$ against $t$ for $N=3$ and two different parameter sets; observed that (a) and (b) show similar qualitative and quantitative dynamics to the corresponding plots in Figure~\ref{fig:examplets}.}%
\label{fig:examplereducedts}%
\end{figure}

\end{document}